\documentclass[11pt]{article}

\usepackage{graphicx}
\usepackage{multirow}
\usepackage{amsmath,amssymb,amsfonts,amsthm}
\usepackage{fullpage}
\usepackage{float}
\usepackage{url}
\usepackage{mathrsfs}
\usepackage{xcolor}
\usepackage{textcomp}
\usepackage{booktabs}
\usepackage{tikz}
\usepackage{pgfplots}
\pgfplotsset{compat=1.18}
\usepackage{hyperref}
\usetikzlibrary{calc,angles,quotes}
\usepackage{caption}
\IfFileExists{bbm.sty}{\usepackage{bbm}\newcommand{\R}{\mathbbm{R}}}{\newcommand{\R}{\mathbb{R}}}
\newcommand{\ip}[2]{\langle #1,\, #2 \rangle}

\newtheorem{theorem}{Theorem}[section]
\newtheorem{lemma}[theorem]{Lemma}
\newtheorem{proposition}[theorem]{Proposition}
\newtheorem{corollary}[theorem]{Corollary}

\newtheorem{remark}[theorem]{Remark}
\providecommand{\coloneqq}{\mathrel{\mathop:}=}

\title{The Method of Ellipcenters for Strongly Convex
  Functions\thanks{The author was supported by the National Science
        Foundation, Grant DMS-2307328, and by an internal
        grant from NIU.}}
\date{}

\begin{document}

\maketitle

\begin{center}
Yunier Bello-Cruz\\
School of Mathematical and Statistical Sciences,
Northern Illinois University\\
DeKalb, IL 60115, USA\\[2pt]
{\tt yunierbello@niu.edu}
\end{center}

\bigskip

\begin{abstract}
\noindent The Method of Ellipcenters (ME) of Behling, Corr\^ea
Aquines, Ferreira Zanatta, and Guigues minimizes a strongly convex
quadratic by moving, at each iteration, to the minimizer of the
objective over the plane spanned by two gradients: one at the current
iterate and one at a companion point on the same level set.
We study this plane-search step for $L$-smooth, $\mu$-strongly convex
functions with condition number $\kappa=L/\mu$.
We prove that every step contracts the optimality gap by the factor
$\rho^2-\sin^2\theta_k/(4\kappa^2)$, with $\rho=(\kappa-1)/(\kappa+1)$
and $\theta_k$ the angle between the two gradients; $\rho^2$ is the
tight factor of steepest descent with exact line search, and no linear
independence assumption is required.
The level-set condition forces the two gradients to lie within the
angle $\arctan(\kappa-1)$ of opposite directions.
If the objective is twice differentiable at the minimizer, the
asymptotic contraction factor is at most the two-step Chebyshev factor
$(\kappa_*-1)^4/(\kappa_*^2+6\kappa_*+1)^2$ of the Hessian at the
solution, and convergence is superlinear when that Hessian has at most
two distinct eigenvalues.
A performance-estimation computation indicates that the worst-case
one-step factor lies strictly between $\rho^4$ and $\rho^2$.
Experiments on regularized logistic regression illustrate the results.


\end{abstract}

\section{Introduction}\label{sec:intro}

The Method of Ellipcenters (ME) was introduced in~\cite{ME2025} for
the unconstrained minimization of a strongly convex quadratic
$f(x)=\tfrac12x^TAx-b^Tx$.
From a non-stationary iterate $x^k$, the method moves along
$-\nabla f(x^k)$ until it returns to the level set of $x^k$, at a
companion point $y^k$, and evaluates a second gradient there.
The two gradients span, from $x^k$, a plane $\Pi_k$ that cuts the
level set in an ellipse orthogonal to $\nabla f(x^k)$ at $x^k$ and to
$\nabla f(y^k)$ at $y^k$; the next iterate is the center of this
ellipse, and the center is the minimizer of $f$ over $\Pi_k$.
The analysis in~\cite{ME2025} shows that each step contracts
$f(x^k)-f(x^*)$ at least by the factor $\rho^2$, with
\begin{align}\label{eq:rho}
  \rho\coloneqq\frac{\kappa-1}{\kappa+1},
  \qquad
  \kappa\coloneqq\frac{\lambda_{\max}(A)}{\lambda_{\min}(A)},
\end{align}
which is the classical factor of steepest descent with exact line
search.
Since $\nabla f(y^k)$ is a combination of $\nabla f(x^k)$ and
$A\nabla f(x^k)$, the plane $\Pi_k$ is the two-dimensional Krylov
space generated at $x^k$, and the factor was improved
in~\cite{BGO2026} to $\rho^4$ and to the Chebyshev value
$\tau^2$, with
\begin{align}\label{eq:tau}
  \tau\coloneqq\frac{(\kappa-1)^2}{\kappa^2+6\kappa+1}.
\end{align}
In classical terms, ME for quadratics is the conjugate gradient
method~\cite{Hestenes1952} restarted every two steps, which is the
optimum $2$-gradient method studied by Forsythe~\cite{Forsythe1968}.

For a non-quadratic $f$ the section of a level set by $\Pi_k$ is no
longer an ellipse, and the two descriptions of the ME iterate, center
of an ellipse and minimizer over the plane, separate.
The authors of ME follow the first one in~\cite{BCFG2026}: they
characterize the centers of all ellipses in $\Pi_k$ that are
orthogonal to the two gradients at $x^k$ and $y^k$, show that these
centers form a half-line starting at the midpoint of $x^k$ and $y^k$,
and choose the next iterate on that half-line, exactly or
approximately.
For $L$-smooth, $\mu$-strongly convex functions they prove the linear
rate $1-\mu^2/L^2$ \cite[Theorem~4.6]{BCFG2026}.
As observed in~\cite{BGO2026}, the resulting method differs from the
one studied here.

This paper follows the second description.
We set
\begin{align}\label{eq:intro_step}
  x^{k+1}\coloneqq\operatorname*{arg\,min}_{x\in\Pi_k}f(x),
  \qquad
  \Pi_k\coloneqq x^k+\operatorname{span}\bigl\{\nabla f(x^k),
  \nabla f(y^k)\bigr\},
\end{align}
for every $L$-smooth, $\mu$-strongly convex $f$, whether the two
gradients are linearly independent or not, and we keep the name ME for
this plane-search method because it coincides with~\cite{ME2025} on
quadratics.
Every point produced from $x^k$ by the variants of~\cite{BCFG2026}
lies in $\Pi_k$, so the step~\eqref{eq:intro_step} decreases $f$ at
least as much as any of them from the same point.
Our results are the following; throughout, $\kappa=L/\mu$ and $\rho$,
$\tau$ are given by~\eqref{eq:rho} and~\eqref{eq:tau} with this
$\kappa$.
\begin{enumerate}
  \item The companion point exists, is unique, and the companion step
    $t_k$ satisfies $2/L\leq t_k\leq 2/\mu$
    (Lemma~\ref{lem:companion}).
    The level-set condition $f(y^k)=f(x^k)$ forces the angle
    $\theta_k$ between the two gradients into the interval
    $[\pi-\arctan(\kappa-1),\pi]$
    (Proposition~\ref{prop:angle}).
  \item Every step satisfies
    $f(x^{k+1})-f(x^*)\leq
    \bigl(\rho^2-\sin^2\theta_k/(4\kappa^2)\bigr)(f(x^k)-f(x^*))$
    (Theorem~\ref{thm:global}).
    The factor $\rho^2$ is the tight worst-case factor of steepest
    descent with exact line search, due to de Klerk, Glineur, and
    Taylor~\cite{deKlerk2017}; the correction term comes from a second
    line search issued from the companion point.
    The bound needs no assumption on the two gradients: when they are
    parallel, the step is an exact line search and the factor is
    $\rho^2$.
  \item If $f$ is twice differentiable at the minimizer $x^*$, then
    the asymptotic contraction factor is at most $\tau_*^2$, the value
    of $\tau^2$ at the condition number $\kappa_*$ of
    $\nabla^2f(x^*)$, and the convergence is superlinear when
    $\nabla^2f(x^*)$ has at most two distinct eigenvalues
    (Theorem~\ref{thm:local}).
    Since $\tau^2<\rho^4$, the asymptotic guarantee for one ME step
    is better than the worst case of two steepest descent steps.
  \item The global picture is different.
    A performance-estimation
    computation~\cite{Drori2014,Taylor2017} indicates that the
    worst-case factor of one ME step over the whole class lies
    strictly between $\rho^4$ and $\rho^2$
    (Section~\ref{sec:pep}); in the worst case one ME step is better
    than one steepest descent step and worse than two.
\end{enumerate}

The step~\eqref{eq:intro_step} requires an exact minimization over a
plane, in the same sense in which steepest descent requires an exact
line search.
For objectives of the form $x\mapsto\phi(Ax)+\tfrac{\mu}{2}\lVert
x\rVert^2$ the restriction of $f$ to $\Pi_k$ is evaluated from two
products with $A$, and the plane search costs $O(m)$ operations per
inner iteration for $A\in\R^{m\times n}$; one ME iteration then costs
two gradients and two products with $A$, twice the cost of a steepest
descent iteration.
Section~\ref{sec:numerical} reports experiments in this setting,
including the variants of~\cite{BCFG2026}, Nesterov's
method~\cite{Nesterov1983,Nesterov2004}, and a nonlinear conjugate
gradient method.

\paragraph{Relation to earlier versions.}
This version replaces the preprints arXiv:2604.21132v1--v3.
The rates stated there, $1-\mu^2/L^2$ in general and
$(\kappa-1)/(\kappa+1)$ at steps with linearly independent gradients,
are superseded by Theorem~\ref{thm:global}.
Three statements of those versions are withdrawn.
Termination in at most two steps for $n=2$ fails: two consecutive
steps with parallel gradients can occur without reaching $x^*$, and we
found such configurations numerically for a quadratic perturbed by
log-cosh terms.
The sufficient conditions for linear independence in terms of
$\nabla^2f(x^k)$ are replaced by Remark~\ref{rem:LD}.
The obtuse angle between the two gradients does not bound
$\sin\theta_k$ from below; Proposition~\ref{prop:angle} bounds it from
above.
The experiments have been redone with a cost count common to all
methods.

Section~\ref{sec:prelim} collects the inequalities we use.
Section~\ref{sec:algorithm} presents the method and its relation to
the quadratic case and to~\cite{BCFG2026}.
Section~\ref{sec:global} contains the geometry of the companion pair,
the global rate, and the performance-estimation computation.
Section~\ref{sec:local} gives the local rate.
Section~\ref{sec:numerical} reports the experiments and
Section~\ref{sec:conclusion} lists open problems.

\section{Preliminaries}\label{sec:prelim}

For $x,y\in\R^n$ we write $\ip{x}{y}=x^Ty$ and
$\lVert x\rVert=\ip{x}{x}^{1/2}$.
For $0<\mu\leq L$, $\mathcal{F}_{\mu,L}(\R^n)$ denotes the class of
continuously differentiable functions $f:\R^n\to\R$ such that
$f-\tfrac{\mu}{2}\lVert\cdot\rVert^2$ is convex and $\nabla f$ is
$L$-Lipschitz continuous.
Every $f\in\mathcal{F}_{\mu,L}(\R^n)$ has a unique minimizer, denoted
by $x^*$, and we write $f^*\coloneqq f(x^*)$.

Throughout the paper $f\in\mathcal{F}_{\mu,L}(\R^n)$ with $0<\mu<L$,
$\kappa\coloneqq L/\mu$, and $\rho$ and $\tau$ are defined
by~\eqref{eq:rho} and~\eqref{eq:tau}.
If $\mu=L$, then $f=f^*+\tfrac{\mu}{2}\lVert\cdot-x^*\rVert^2$ and
every method considered here terminates in one step.

\begin{proposition}[Standard inequalities]\label{prop:standard}
For all $x,y\in\R^n$,
\begin{align}
  \frac{\mu}{2}\lVert y-x\rVert^2
  &\leq f(y)-f(x)-\ip{\nabla f(x)}{y-x}
  \leq\frac{L}{2}\lVert y-x\rVert^2,
  \label{eq:sandwich}\\
  2\mu\,(f(x)-f^*)
  &\leq\lVert\nabla f(x)\rVert^2
  \leq 2L\,(f(x)-f^*),
  \label{eq:PL}\\
  \frac{\mu}{2}\lVert x-x^*\rVert^2
  &\leq f(x)-f^*
  \leq\frac{L}{2}\lVert x-x^*\rVert^2,
  \label{eq:growth}\\
  f\Bigl(\frac{x+y}{2}\Bigr)
  &\leq\frac{f(x)+f(y)}{2}-\frac{\mu}{8}\lVert x-y\rVert^2,
  \label{eq:midpoint}\\
  \ip{\nabla f(x)-\nabla f(y)}{x-y}
  &\geq\frac{\mu L}{\mu+L}\lVert x-y\rVert^2
  +\frac{1}{\mu+L}\lVert\nabla f(x)-\nabla f(y)\rVert^2.
  \label{eq:coco}
\end{align}
\end{proposition}

\begin{proof}
The left inequality in~\eqref{eq:sandwich} is equivalent to the
convexity of $f-\tfrac{\mu}{2}\lVert\cdot\rVert^2$ and the right one
is the descent lemma; see~\cite[Section~2.1]{Nesterov2004}.
Reading~\eqref{eq:sandwich} as two bounds on $f(y)$ and minimizing
over $y$ gives
$f(x)-\lVert\nabla f(x)\rVert^2/(2\mu)\leq f^*\leq
f(x)-\lVert\nabla f(x)\rVert^2/(2L)$, which is~\eqref{eq:PL}; the
left inequality in~\eqref{eq:PL} is the Polyak--\L{}ojasiewicz
inequality~\cite{Polyak1963}.
Taking $x=x^*$ in~\eqref{eq:sandwich} gives~\eqref{eq:growth}.
Inequality~\eqref{eq:midpoint} is the midpoint convexity of
$f-\tfrac{\mu}{2}\lVert\cdot\rVert^2$ combined with the identity
$\tfrac12\lVert x\rVert^2+\tfrac12\lVert y\rVert^2
-\lVert\tfrac12(x+y)\rVert^2=\tfrac14\lVert x-y\rVert^2$.
Inequality~\eqref{eq:coco}
is~\cite[Theorem~2.1.12]{Nesterov2004}.
\end{proof}

The next result is the starting point of the global analysis.

\begin{proposition}[Exact line search
  {\cite[Theorem~1.2]{deKlerk2017}}]\label{prop:DGT}
Let $u\in\R^n$ and let $u^+$ be the minimizer of $f$ over the line
$\{u-s\nabla f(u):s\in\R\}$.
Then
\begin{align}\label{eq:DGT}
  f(u^+)-f^*\leq\rho^2\,\bigl(f(u)-f^*\bigr).
\end{align}
\end{proposition}

The factor $\rho^2$ in~\eqref{eq:DGT} cannot be improved: equality
holds at every iteration of steepest descent for a quadratic in $\R^2$
with eigenvalues $\mu$ and $L$ and a suitable starting
point~\cite[Example~1.3]{deKlerk2017}.

\section{The method}\label{sec:algorithm}

\subsection{The companion point}\label{sec:companion}

\begin{lemma}[Companion step]\label{lem:companion}
Let $x\in\R^n$ with $v\coloneqq\nabla f(x)\neq 0$ and set
$\varphi(s)\coloneqq f(x-sv)$ for $s\in\R$.
\begin{enumerate}
  \item\label{item:comp_t}
    There is a unique $t>0$ such that $f(x-tv)=f(x)$, and
    \begin{align}\label{eq:t_bounds}
      \frac{2}{L}\leq t\leq\frac{2}{\mu}.
    \end{align}
  \item\label{item:comp_s}
    The minimizer $\bar s$ of $\varphi$ over $\R$ is unique and
    satisfies $0<\bar s<t$ and
    \begin{align}\label{eq:s_bounds}
      \frac{1}{L}\leq\bar s\leq\frac{1}{\mu}.
    \end{align}
\end{enumerate}
\end{lemma}

\begin{proof}
By~\eqref{eq:sandwich} with $y=x-sv$,
\begin{align}\label{eq:phi_models}
  -s\lVert v\rVert^2+\frac{\mu}{2}s^2\lVert v\rVert^2
  \leq\varphi(s)-\varphi(0)
  \leq-s\lVert v\rVert^2+\frac{L}{2}s^2\lVert v\rVert^2
  \qquad\text{for all }s\in\R.
\end{align}
Hence $\varphi(s)>\varphi(0)$ for $s<0$ and for $s>2/\mu$, and
$\varphi(s)<\varphi(0)$ for $0<s<2/L$.
The function $\varphi$ is continuous and strictly convex, so
$\{s:\varphi(s)\leq\varphi(0)\}$ is a compact interval $[0,t]$ with
$2/L\leq t\leq 2/\mu$, $\varphi(t)=\varphi(0)$, and
$\varphi<\varphi(0)$ on $(0,t)$.
Outside $[0,t]$ we have $\varphi>\varphi(0)$, so $t$ is the only
positive solution of $\varphi(s)=\varphi(0)$.
This proves~\ref{item:comp_t}.

The minimizer $\bar s$ exists and is unique by strong convexity, and
it lies in $(0,t)$ because $\varphi(0)=\varphi(t)$ and
$\varphi<\varphi(0)$ in between.
For $s>0$,
\begin{align}\label{eq:phi_prime}
  \varphi'(s)-\varphi'(0)
  =\frac{1}{s}\ip{\nabla f(x-sv)-\nabla f(x)}{(x-sv)-x}
  \in\bigl[\mu s\lVert v\rVert^2,\,L s\lVert v\rVert^2\bigr],
\end{align}
by the Lipschitz continuity of $\nabla f$ and its strong monotonicity,
which follows from~\eqref{eq:coco}.
Since $\varphi'(0)=-\lVert v\rVert^2$ and $\varphi'(\bar s)=0$,
\eqref{eq:phi_prime} at $s=\bar s$ gives
$\mu\bar s\leq 1\leq L\bar s$, which is~\eqref{eq:s_bounds}.
\end{proof}

Existence and uniqueness of $t$ for differentiable strongly convex
functions were already noted in the concluding section
of~\cite{ME2025}; the lower bound in~\eqref{eq:t_bounds} is also used
in~\cite{BCFG2026}.

\subsection{Algorithm}\label{sec:alg_desc}

For a non-stationary $x^k$ we write
\begin{align}\label{eq:notation_k}
  v^k\coloneqq\nabla f(x^k),\qquad
  y^k\coloneqq x^k-t_kv^k,\qquad
  w^k\coloneqq\nabla f(y^k),
\end{align}
with $t_k$ the companion step of Lemma~\ref{lem:companion} at $x^k$,
and
\begin{align}\label{eq:Pi_k}
  \Pi_k\coloneqq\bigl\{x^k+\alpha v^k+\beta w^k
    :(\alpha,\beta)\in\R^2\bigr\}.
\end{align}
Since $f(y^k)=f(x^k)>f^*$, the vector $w^k$ is nonzero.
We denote by $\theta_k\in[0,\pi]$ the angle between $v^k$ and $w^k$.

\smallskip
\noindent\rule[0.5ex]{\columnwidth}{1pt}\\[-2pt]
\textbf{Algorithm ME.}
Method of Ellipcenters with plane search.\\[-6pt]
\noindent\rule[0.5ex]{\columnwidth}{0.4pt}

\noindent\textbf{Input:} $x^1\in\R^n$; set $k=1$.

\smallskip
\noindent\textbf{Step~1.}
Compute $v^k=\nabla f(x^k)$.
If $v^k=0$, \textbf{stop}.

\smallskip
\noindent\textbf{Step~2.}
Find $t_k>0$ with $f(x^k-t_kv^k)=f(x^k)$.
Set $y^k=x^k-t_kv^k$ and compute $w^k=\nabla f(y^k)$.

\smallskip
\noindent\textbf{Step~3.}
Set $x^{k+1}=\operatorname*{arg\,min}\{f(x):x\in\Pi_k\}$.

\smallskip
\noindent\textbf{Step~4.}
Set $k\leftarrow k+1$ and go to Step~1.\\[-6pt]
\noindent\rule[0.5ex]{\columnwidth}{1pt}
\smallskip

The affine set $\Pi_k$ is a plane when $v^k$ and $w^k$ are linearly
independent and the line through $x^k$ and $y^k$ otherwise; in both
cases $f$ has a unique minimizer on it, so Step~3 is well defined.
In the second case $x^{k+1}$ is the exact line search point along
$-v^k$, which lies strictly between $x^k$ and $y^k$ by
Lemma~\ref{lem:companion}(\ref{item:comp_s}).
The optimality conditions of Step~3 are
\begin{align}\label{eq:orth}
  \ip{\nabla f(x^{k+1})}{v^k}=0,
  \qquad
  \ip{\nabla f(x^{k+1})}{w^k}=0,
\end{align}
and, since $x^{k+1}-u\in\operatorname{span}\{v^k,w^k\}$ for every
$u\in\Pi_k$, they give
$\ip{\nabla f(x^{k+1})}{x^{k+1}-u}=0$ for all $u\in\Pi_k$.

\begin{remark}[Quadratic objectives]\label{rem:quadratic}
Let $f(x)=\tfrac12x^TAx-b^Tx$ with $A$ symmetric positive definite.
Then $\varphi$ in Lemma~\ref{lem:companion} is a parabola and
\begin{align}\label{eq:tk_quad}
  t_k=\frac{2\lVert v^k\rVert^2}{\ip{v^k}{Av^k}},
\end{align}
twice the Cauchy step, as in~\cite{ME2025}.
Moreover $w^k=v^k-t_kAv^k$, so
$\operatorname{span}\{v^k,w^k\}=\operatorname{span}\{v^k,Av^k\}$ and
$x^{k+1}$ is the second iterate of the conjugate gradient
method~\cite{Hestenes1952} started at $x^k$.
Algorithm~ME is then the conjugate gradient method restarted every two
steps, that is, Forsythe's optimum $2$-gradient
method~\cite{Forsythe1968}, and it coincides with the method
of~\cite{ME2025}, in which $x^{k+1}$ is the center of the ellipse
$\{f=f(x^k)\}\cap\Pi_k$.
The vectors $v^k$ and $w^k$ are linearly dependent exactly when $v^k$
is an eigenvector of $A$; in that case $\Pi_k$ is a line and
$x^{k+1}=\tfrac12(x^k+y^k)=x^*$.
For quadratics the one-step factor $\rho^2$ was proved
in~\cite{ME2025}, and the factors $\rho^4$ and $\tau^2$
in~\cite{BGO2026}.
\end{remark}

\begin{remark}[The centers variant]\label{rem:centers}
For non-quadratic $f$ the method of~\cite{BCFG2026} takes
$x^{k+1}=\tfrac12(x^k+y^k)+s_kd^k$ with $s_k\geq 0$ and a direction
$d^k\in\operatorname{span}\{v^k,w^k\}$ along which the centers of the
admissible ellipses lie, and $x^{k+1}=\tfrac12(x^k+y^k)$ when $v^k$
and $w^k$ are linearly dependent.
All these points belong to $\Pi_k$.
Consequently, if $\tilde x$ is the point produced from $x^k$ by any
variant of~\cite{BCFG2026}, then the point $x^{k+1}$ of Algorithm~ME
satisfies $f(x^{k+1})\leq f(\tilde x)$.
The comparison is between single steps from a common point; it does
not order the two sequences.
\end{remark}

\begin{remark}[Parallel gradients]\label{rem:LD}
Set $h^k\coloneqq(v^k-w^k)/t_k$.
The vectors $v^k$ and $w^k$ are linearly dependent if and only if
$h^k$ is a multiple of $v^k$.
If $f$ is twice continuously differentiable, then
$h^k=\bar H_kv^k$ with the averaged Hessian
$\bar H_k\coloneqq\int_0^1\nabla^2f(x^k-st_kv^k)\,ds$, so linear
dependence means that $v^k$ is an eigenvector of $\bar H_k$.
It can hold at points at which $v^k$ is not an eigenvector of
$\nabla^2f(x^k)$, and for $n=2$ it typically holds along curves
through $x^*$.
None of the results below assumes linear independence.
\end{remark}

\begin{remark}[Cost of one iteration]\label{rem:cost}
One iteration uses two gradients, a one-dimensional root search for
$t_k$, and a minimization in at most two variables.
Suppose that $f(x)=\phi(Ax)+\tfrac{\mu}{2}\lVert x\rVert^2$ with
$A\in\R^{m\times n}$ and $\phi:\R^m\to\R$ separable.
If $Ax^k$ is stored, then $Av^k$ and $Aw^k$ give
$A(x^k+\alpha v^k+\beta w^k)$ for all $(\alpha,\beta)$ at no further
cost, and $f$, its gradient, and its Hessian with respect to
$(\alpha,\beta)$ are computed in $O(m)$ operations.
One iteration then costs two products with $A^T$, two products with
$A$, and $O(m+n)$ operations per inner evaluation.
Steepest descent with exact line search costs one product with $A^T$
and one with $A$ per iteration in the same model.
\end{remark}

\section{Global convergence}\label{sec:global}

\subsection{Geometry of the companion pair}\label{sec:geometry}

\begin{proposition}[Companion gradients]\label{prop:angle}
Let $x^k$ be non-stationary and let $P_k\coloneqq I-v^k(v^k)^T/
\lVert v^k\rVert^2$ be the orthogonal projector onto the orthogonal
complement of $v^k$.
Then
\begin{align}
  \ip{v^k}{w^k}
  &\leq-\frac{\mu t_k}{2}\lVert v^k\rVert^2
  \leq-\frac{1}{\kappa}\lVert v^k\rVert^2,
  \label{eq:obtuse}\\
  \lVert P_kw^k\rVert
  &\leq\frac{(L-\mu)\,t_k}{2}\lVert v^k\rVert.
  \label{eq:perp}
\end{align}
Consequently
\begin{align}\label{eq:angle}
  \pi-\arctan(\kappa-1)\leq\theta_k\leq\pi,
  \qquad\text{that is,}\qquad
  \sin\theta_k\leq(\kappa-1)\,\lvert\cos\theta_k\rvert
  \ \text{ and }\ \cos\theta_k<0.
\end{align}
\end{proposition}

\begin{proof}
We drop the index $k$.
The left inequality in~\eqref{eq:sandwich} at the pair $(y,x)$ reads
$f(x)\geq f(y)+\ip{w}{x-y}+\tfrac{\mu}{2}\lVert x-y\rVert^2$.
Since $f(x)=f(y)$ and $x-y=tv$, we get
$0\geq t\ip{w}{v}+\tfrac{\mu}{2}t^2\lVert v\rVert^2$, which is the
first inequality in~\eqref{eq:obtuse}; the second one follows from
$t\geq 2/L$.

Set $h\coloneqq(v-w)/t$.
Inequality~\eqref{eq:coco} at the pair $(x,y)$, divided by $t^2$,
gives $(\mu+L)\ip{h}{v}\geq\mu L\lVert v\rVert^2+\lVert h\rVert^2$,
that is,
\begin{align}\label{eq:ball}
  \Bigl\lVert h-\frac{\mu+L}{2}\,v\Bigr\rVert
  \leq\frac{L-\mu}{2}\lVert v\rVert.
\end{align}
Since $Pv=0$ and $P$ is an orthogonal projector,
$\lVert Ph\rVert=\lVert P(h-\tfrac{\mu+L}{2}v)\rVert
\leq\tfrac{L-\mu}{2}\lVert v\rVert$, and $Pw=-tPh$
gives~\eqref{eq:perp}.

By~\eqref{eq:obtuse}, $\cos\theta<0$, so
$\pi-\theta\in[0,\pi/2)$ and
\begin{align}
  \tan(\pi-\theta)
  =\frac{\lVert Pw\rVert}{\lvert\ip{v}{w}\rvert/\lVert v\rVert}
  \leq\frac{(L-\mu)t\lVert v\rVert/2}{\mu t\lVert v\rVert/2}
  =\kappa-1,
\end{align}
which is~\eqref{eq:angle}.
\end{proof}

The two gradients therefore always form an obtuse angle, and they are
close to opposite when $\kappa$ is close to $1$.
The quantity that enters the rate below is $\sin\theta_k$, which
vanishes exactly when $v^k$ and $w^k$ are linearly dependent and is at
most $(\kappa-1)/\sqrt{1+(\kappa-1)^2}$ by~\eqref{eq:angle}.

\subsection{The global rate}\label{sec:rate}

\begin{theorem}[Global linear rate]\label{thm:global}
Let $f\in\mathcal{F}_{\mu,L}(\R^n)$ with $0<\mu<L$ and let $x^k$ be a
non-stationary iterate of Algorithm~ME.
Then
\begin{align}\label{eq:global_rate}
  f(x^{k+1})-f^*
  \leq\Bigl(\rho^2-\frac{\sin^2\theta_k}{4\kappa^2}\Bigr)
  \bigl(f(x^k)-f^*\bigr).
\end{align}
\end{theorem}

\begin{proof}
Write $\Delta_k\coloneqq f(x^k)-f^*$.
Let $\hat x^k$ be the minimizer of $f$ over the line
$\{x^k-sv^k:s\in\R\}$ and let $\hat z^k=y^k-s_kw^k$ be the minimizer
of $f$ over the line $\{y^k-sw^k:s\in\R\}$.
Proposition~\ref{prop:DGT} applied at $u=x^k$ and at $u=y^k$, together
with $f(y^k)=f(x^k)$, gives
\begin{align}\label{eq:two_ls}
  f(\hat x^k)-f^*\leq\rho^2\Delta_k,
  \qquad
  f(\hat z^k)-f^*\leq\rho^2\Delta_k.
\end{align}
Both points lie in $\Pi_k$, and so does their midpoint $m^k$.
Since $x^{k+1}$ minimizes $f$ over $\Pi_k$, inequalities
\eqref{eq:midpoint} and~\eqref{eq:two_ls} yield
\begin{align}\label{eq:mid_step}
  f(x^{k+1})-f^*
  \leq f(m^k)-f^*
  \leq\rho^2\Delta_k-\frac{\mu}{8}\lVert\hat x^k-\hat z^k\rVert^2.
\end{align}
It remains to bound the distance between the two line search points
from below.
With $\hat x^k=x^k-\bar s_kv^k$ we have
$\hat x^k-\hat z^k=(t_k-\bar s_k)v^k+s_kw^k$, and the projector $P_k$
of Proposition~\ref{prop:angle} removes the first term:
\begin{align}\label{eq:sep}
  \lVert\hat x^k-\hat z^k\rVert
  \geq\lVert P_k(\hat x^k-\hat z^k)\rVert
  =s_k\lVert P_kw^k\rVert
  =s_k\lVert w^k\rVert\sin\theta_k.
\end{align}
Lemma~\ref{lem:companion}(\ref{item:comp_s}) at the point $y^k$ gives
$s_k\geq 1/L$, and~\eqref{eq:PL} at $y^k$ gives
$\lVert w^k\rVert^2\geq 2\mu(f(y^k)-f^*)=2\mu\Delta_k$.
Hence
$\lVert\hat x^k-\hat z^k\rVert^2
\geq 2\mu\sin^2\theta_k\,\Delta_k/L^2$,
and~\eqref{eq:mid_step} becomes
\begin{align}
  f(x^{k+1})-f^*
  \leq\Bigl(\rho^2-\frac{\mu^2\sin^2\theta_k}{4L^2}\Bigr)\Delta_k,
\end{align}
which is~\eqref{eq:global_rate}.
\end{proof}

When $v^k$ and $w^k$ are linearly dependent, $\sin\theta_k=0$,
$x^{k+1}=\hat x^k$, and~\eqref{eq:global_rate} is
Proposition~\ref{prop:DGT}.
The factor in~\eqref{eq:global_rate} is nonnegative for every
admissible angle: by~\eqref{eq:angle},
$\sin^2\theta_k/(4\kappa^2)\leq(\kappa-1)^2/
\bigl(4\kappa^2(1+(\kappa-1)^2)\bigr)\leq\rho^2$, the last inequality
being equivalent to $4u^4+8u^3+7u^2+4u\geq 0$ with $u=\kappa-1$.

\begin{corollary}\label{cor:global}
Let $(x^k)_{k\geq 1}$ be generated by Algorithm~ME.
\begin{enumerate}
  \item\label{item:cor_f}
    $f(x^{k+1})-f^*\leq\rho^{2k}\,(f(x^1)-f^*)$ and
    $\lVert x^{k+1}-x^*\rVert^2\leq\kappa\,\rho^{2k}
    \lVert x^1-x^*\rVert^2$ for all $k\geq 1$.
  \item\label{item:cor_c}
    If $\sin^2\theta_k\geq c>0$ for all $k$, then
    $f(x^{k+1})-f^*\leq(\rho^2-c/(4\kappa^2))^k\,(f(x^1)-f^*)$.
  \item\label{item:cor_dom}
    $f(x^{k+1})\leq f(u)$ for every $u\in\Pi_k$; in particular
    $f(x^{k+1})\leq f(\hat x^k)$, with $\hat x^k$ the exact line
    search point from $x^k$, and $f(x^{k+1})\leq f(\tilde x)$ for
    the points $\tilde x$ of Remark~\ref{rem:centers}.
\end{enumerate}
\end{corollary}

\begin{proof}
Item~\ref{item:cor_f} follows from~\eqref{eq:global_rate}
and~\eqref{eq:growth}; item~\ref{item:cor_c}
from~\eqref{eq:global_rate}; item~\ref{item:cor_dom} from the
definition of $x^{k+1}$.
\end{proof}

Since $\rho^2=1-4\kappa/(\kappa+1)^2$, the bound
in~\ref{item:cor_f} yields $f(x^k)-f^*\leq\varepsilon$ after
$O(\kappa\log(1/\varepsilon))$ iterations, against
$O(\kappa^2\log(1/\varepsilon))$ for the rate $1-1/\kappa^2$
of~\cite{BCFG2026}.
The correction term in~\eqref{eq:global_rate} is at most
$1/(4\kappa^2)$, small compared with $1-\rho^2$.
The next subsection shows that the true worst-case gain over $\rho^2$
is larger.

\subsection{The worst-case factor of one step}\label{sec:pep}

Let $B(\kappa)$ be the supremum of
$(f(x^+)-f^*)/(f(x)-f^*)$ over all $n\geq 1$, all
$f\in\mathcal{F}_{\mu,L}(\R^n)$, and all non-stationary $x$, with
$x^+$ the ME step from $x$; by scaling, it depends on $\mu$ and $L$
through $\kappa$ only.
Theorem~\ref{thm:global} gives $B(\kappa)\leq\rho^2$.
We estimate $B(\kappa)$ with the performance-estimation methodology
of~\cite{Drori2014,Taylor2017}, in the form used
in~\cite{deKlerk2017} for the exact line search.

Fix $t\in[2/L,2/\mu]$.
Consider the four points $x$, $y=x-tv$, $x^+$, $x^*$, with gradients
$v$, $w$, $g^+$, $0$ and function values $f_x$, $f_x$, $f_+$, $f^*$.
By~\cite[Theorem~4]{Taylor2017}, these data are consistent with some
$f\in\mathcal{F}_{\mu,L}$ if and only if, for every ordered pair of
distinct points $(x_i,g_i,f_i)$, $(x_j,g_j,f_j)$ among them,
\begin{align}\label{eq:interp}
  f_i\geq f_j+\ip{g_j}{x_i-x_j}
  +\frac{1}{2(1-\mu/L)}\Bigl(\frac1L\lVert g_i-g_j\rVert^2
  +\mu\lVert x_i-x_j\rVert^2
  -\frac{2\mu}{L}\ip{g_j-g_i}{x_j-x_i}\Bigr).
\end{align}
These twelve inequalities are linear in $f_x$, $f_+$ and in the Gram
matrix $G\succeq 0$ of the five vectors $x-x^*$, $v$, $w$, $g^+$,
$x^+-x$.
The ME step satisfies in addition
\begin{align}\label{eq:pep_orth}
  \ip{g^+}{v}=0,\qquad\ip{g^+}{w}=0,\qquad\ip{g^+}{x^+-x}=0,
\end{align}
by~\eqref{eq:orth}.
Let $V(\kappa,t)$ be the optimal value of the semidefinite program
that maximizes $f_+-f^*$ subject to~\eqref{eq:interp},
\eqref{eq:pep_orth}, $G\succeq 0$, and $f_x-f^*=1$.
Every ME step with companion step $t$ is feasible for this program, so
\begin{align}\label{eq:Bbar}
  B(\kappa)\leq\bar B(\kappa)\coloneqq
  \sup_{2/L\leq t\leq 2/\mu}V(\kappa,t).
\end{align}
The program is a relaxation because~\eqref{eq:pep_orth} does not
force $x^+-x$ into $\operatorname{span}\{v,w\}$.

\begin{table}[H]
\centering
{\small
\begin{tabular}{rcccc}
\toprule
$\kappa$ & $\tau^2$ & $\rho^4$ & $\bar B(\kappa)$ & $\rho^2$\\
\midrule
$2$    & $0.0035$ & $0.0123$ & $0.0960$ & $0.1111$\\
$5$    & $0.0816$ & $0.1975$ & $0.3541$ & $0.4444$\\
$10$   & $0.2531$ & $0.4481$ & $0.5366$ & $0.6694$\\
$20$   & $0.4801$ & $0.6701$ & $0.7230$ & $0.8186$\\
$50$   & $0.7348$ & $0.8521$ & $0.8780$ & $0.9231$\\
$100$  & $0.8548$ & $0.9231$ & $0.9375$ & $0.9608$\\
$200$  & $0.9238$ & $0.9608$ & $0.9686$ & $0.9802$\\
$500$  & $0.9686$ & $0.9841$ & $0.9875$ & $0.9920$\\
$1000$ & $0.9842$ & $0.9920$ & $0.9938$ & $0.9960$\\
\bottomrule
\end{tabular}
}
\vspace{4pt}
\caption{Numerical values of the bound $\bar B(\kappa)$
  in~\eqref{eq:Bbar} for the worst-case factor of one ME step,
  between the factor $\rho^2$ of one exact line search step, the
  factor $\rho^4$ of two, and the two-step Chebyshev factor
  $\tau^2$.}
\label{tab:pep}
\end{table}

Table~\ref{tab:pep} reports $\bar B(\kappa)$ computed with the
interior-point solver Clarabel~\cite{Goulart2026} through
CVXPY~\cite{Diamond2016}; the supremum over $t$ was taken on a
geometric grid of $301$ points in $[2/L,2/\mu]$ and refined by a
scalar search around the best grid point.
The values are numerical and are not used in any proof.
Three observations can be read from the computation.
First, the same program without the companion point returns $\rho^2$
to within $10^{-10}$, in agreement with
Proposition~\ref{prop:DGT}.
Second, $\rho^4<\bar B(\kappa)<\rho^2$ for all tested $\kappa$, and
$(1-\bar B(\kappa))/(1-\rho^2)$ lies between $1.5$ and $1.6$ for
$\kappa\geq 20$.
Third, in every computed optimal solution the vector $x^+-x$ lies in
$\operatorname{span}\{v,w\}$ up to a relative squared distance below
$10^{-8}$; such a solution is the Gram matrix of a genuine ME step for
a function of $\mathcal{F}_{\mu,L}$ interpolating the
data~\cite[Theorem~4]{Taylor2017}, which indicates
$B(\kappa)=\bar B(\kappa)$.
For $\kappa\leq 10$ the maximizing step is $t=2/\sqrt{\mu L}$ and the
optimal configuration has parallel gradients; for $\kappa\geq 20$ the
gradients are linearly independent at the optimum.

The computation supports two conclusions.
One ME step has a strictly better worst-case factor than one exact
line search step, by a margin that Theorem~\ref{thm:global} does not
capture when the two gradients are parallel.
On the other hand, one ME step, which uses two gradients, has a worse
worst-case factor than two steepest descent steps.
A closed form for $B(\kappa)$ and a proof of these two inequalities
are open.

\section{Local convergence}\label{sec:local}

For a symmetric positive definite matrix $H$ with spectrum
$\sigma(H)$, let $\mathcal{P}_2$ be the set of real polynomials $p$ of
degree at most two with $p(0)=1$, and
\begin{align}\label{eq:pi2}
  \pi_2(H)\coloneqq\min_{p\in\mathcal{P}_2}\,
  \max_{\lambda\in\sigma(H)}p(\lambda)^2.
\end{align}
The minimum is attained, and $\pi_2(H)=0$ if and only if $H$ has at
most two distinct eigenvalues.
If $\sigma(H)\subset[a,b]$ with $0<a<b$, the polynomial
$p(\lambda)=T_2\bigl((b+a-2\lambda)/(b-a)\bigr)/
T_2\bigl((b+a)/(b-a)\bigr)$, with $T_2(s)=2s^2-1$ the Chebyshev
polynomial of degree two, belongs to $\mathcal{P}_2$ and satisfies
$\lvert p\rvert\leq 1/T_2\bigl((b+a)/(b-a)\bigr)$ on $[a,b]$.
Therefore
\begin{align}\label{eq:pi2_cheb}
  \pi_2(H)\leq\tau(\kappa_H)^2,
  \qquad
  \tau(s)\coloneqq\frac{(s-1)^2}{s^2+6s+1},
  \qquad
  \kappa_H\coloneqq\frac{\lambda_{\max}(H)}{\lambda_{\min}(H)}.
\end{align}
We also write $\rho(s)\coloneqq(s-1)/(s+1)$, so that $\rho=\rho(\kappa)$
and $\tau=\tau(\kappa)$ in~\eqref{eq:rho} and~\eqref{eq:tau}.
For a quadratic with Hessian $A$, the optimality of the conjugate
gradient iterates over Krylov spaces and
Remark~\ref{rem:quadratic} give
$f(x^{k+1})-f^*\leq\pi_2(A)\,(f(x^k)-f^*)$ at every step.
The next theorem shows that the same factor governs the tail of the
ME sequence for every objective that is twice differentiable at the
solution.

\begin{theorem}[Local rate]\label{thm:local}
Let $f\in\mathcal{F}_{\mu,L}(\R^n)$ with $0<\mu<L$, and suppose that
$\nabla f$ is differentiable at $x^*$, with derivative
$H_*\coloneqq\nabla^2f(x^*)$.
Let $(x^k)_{k\geq 1}$ be generated by Algorithm~ME with
$x^k\neq x^*$ for all $k$.
Then
\begin{align}\label{eq:local_rate}
  \limsup_{k\to\infty}\frac{f(x^{k+1})-f^*}{f(x^k)-f^*}
  \leq\pi_2(H_*)
  \leq\tau(\kappa_*)^2
  =\frac{(\kappa_*-1)^4}{(\kappa_*^2+6\kappa_*+1)^2},
\end{align}
with $\kappa_*\coloneqq\lambda_{\max}(H_*)/\lambda_{\min}(H_*)
\leq\kappa$.
In particular, if $H_*$ has at most two distinct eigenvalues, then
$(f(x^k)-f^*)_{k\geq 1}$ converges $Q$-superlinearly to zero.
\end{theorem}

\begin{proof}
The matrix $H_*$ is symmetric, as is every second derivative at a
point, and strong monotonicity and Lipschitz continuity of $\nabla f$
give $\mu I\preceq H_*\preceq LI$, so
$\kappa_*\leq\kappa$.
Write $\nabla f(x^*+d)=H_*d+R(d)$ and
$\omega(s)\coloneqq\sup\{\lVert R(d)\rVert/\lVert d\rVert:
0<\lVert d\rVert\leq s\}$; the function $\omega$ is nondecreasing and
$\omega(s)\to 0$ as $s\to 0$.
Integrating $s\mapsto\ip{\nabla f(x^*+sd)}{d}$ over $[0,1]$ gives
\begin{align}\label{eq:taylor2}
  \Bigl\lvert f(x^*+d)-f^*-\tfrac12\ip{H_*d}{d}\Bigr\rvert
  \leq\tfrac12\,\omega(\lVert d\rVert)\,\lVert d\rVert^2
  \qquad\text{for all }d\in\R^n.
\end{align}
We use the norm $\lVert d\rVert_*\coloneqq\ip{H_*d}{d}^{1/2}$, which
satisfies $\sqrt{\mu}\lVert d\rVert\leq\lVert d\rVert_*\leq
\sqrt{L}\lVert d\rVert$.

Fix $p\in\mathcal{P}_2$, $p(\lambda)=1-a\lambda-b\lambda^2$, and set
$m_p\coloneqq\max_{\lambda\in\sigma(H_*)}\lvert p(\lambda)\rvert$ and
$M_p\coloneqq\max_{\lambda\in[\mu,L]}\lvert p(\lambda)\rvert$.
Fix $k$ and write $e\coloneqq x^k-x^*$, $r\coloneqq\lVert e\rVert$,
and $v$, $w$, $t$, $y$ for $v^k$, $w^k$, $t_k$, $y^k$.
By~\eqref{eq:growth} and $f(y)=f(x^k)$,
$\lVert y-x^*\rVert\leq\sqrt{\kappa}\,r$.
From $v=H_*e+R(e)$ and
$w=H_*(e-tv)+R(y-x^*)$ we obtain, for $h\coloneqq(v-w)/t$,
\begin{align}\label{eq:h_est}
  h-H_*v=\frac{R(e)-R(y-x^*)}{t},
  \qquad
  \lVert h-H_*v\rVert
  \leq\frac{L}{2}\bigl(\omega(r)\,r
  +\omega(\sqrt{\kappa}r)\sqrt{\kappa}\,r\bigr)
  \leq L\sqrt{\kappa}\,\delta\,r,
\end{align}
with $\delta\coloneqq\omega(\sqrt{\kappa}\,r)$; we used
$t\geq 2/L$ from~\eqref{eq:t_bounds}.
The point $u\coloneqq x^k-av-bh$ belongs to $\Pi_k$, and
\begin{align}\label{eq:u_est}
  u-x^*=p(H_*)e-aR(e)-bH_*R(e)-b\,(h-H_*v),
  \qquad
  \lVert u-x^*-p(H_*)e\rVert\leq C_p\,\delta\,r,
\end{align}
with $C_p\coloneqq\lvert a\rvert+\lvert b\rvert L(1+\sqrt{\kappa})$.
Since $p(H_*)$ commutes with $H_*$,
$\lVert p(H_*)e\rVert_*\leq m_p\lVert e\rVert_*$ and
$\lVert p(H_*)e\rVert\leq M_pr$.
Hence, using $r\leq\lVert e\rVert_*/\sqrt{\mu}$,
\begin{align}\label{eq:u_norms}
  \lVert u-x^*\rVert_*
  \leq\bigl(m_p+\sqrt{\kappa}\,C_p\,\delta\bigr)\lVert e\rVert_*,
  \qquad
  \lVert u-x^*\rVert\leq(M_p+C_p\,\delta)\,r.
\end{align}
Since $x^{k+1}$ minimizes $f$ over $\Pi_k$, \eqref{eq:taylor2} at
$d=u-x^*$ and~\eqref{eq:u_norms} give
\begin{align}\label{eq:num}
  f(x^{k+1})-f^*
  \leq f(u)-f^*
  \leq\frac12\Bigl[\bigl(m_p+\sqrt{\kappa}\,C_p\,\delta\bigr)^2
  +\frac{\omega\bigl((M_p+C_p\delta)r\bigr)(M_p+C_p\delta)^2}{\mu}
  \Bigr]\lVert e\rVert_*^2,
\end{align}
and~\eqref{eq:taylor2} at $d=e$ gives
\begin{align}\label{eq:den}
  f(x^k)-f^*\geq\frac12\Bigl(1-\frac{\omega(r)}{\mu}\Bigr)
  \lVert e\rVert_*^2.
\end{align}
By Corollary~\ref{cor:global}, $r=\lVert x^k-x^*\rVert\to 0$ as
$k\to\infty$, so $\delta\to 0$ and all the terms containing $\omega$
in~\eqref{eq:num} and~\eqref{eq:den} tend to zero.
Dividing~\eqref{eq:num} by~\eqref{eq:den} and letting $k\to\infty$
gives
$\limsup_{k}(f(x^{k+1})-f^*)/(f(x^k)-f^*)\leq m_p^2$.
Taking the minimum over $p\in\mathcal{P}_2$ gives the first
inequality in~\eqref{eq:local_rate}; the second one
is~\eqref{eq:pi2_cheb}.
If $H_*$ has at most two distinct eigenvalues, then $\pi_2(H_*)=0$
and the ratio tends to zero.
\end{proof}

Theorem~\ref{thm:local} uses only the inclusion
$x^k-av^k-bh^k\in\Pi_k$; it requires no linear independence, since
this point lies on the line $\Pi_k$ when the gradients are parallel.

\begin{remark}[Comparison with steepest descent]\label{rem:compare_sd}
For a quadratic with Hessian $A$ and a suitable starting point,
steepest descent with exact line search contracts $f-f^*$ by exactly
$\rho(\kappa_A)^2$ at every step~\cite[Example~1.3]{deKlerk2017}, so
two steps give $\rho(\kappa_A)^4$; its asymptotic behavior for general
starting points is described in~\cite{Akaike1959}.
Since $s^2+6s+1>(s+1)^2$ for $s>0$, we have $\tau(s)<\rho(s)^2$ and
\begin{align}
  \tau(s)^2<\rho(s)^4\quad\text{for all }s>1,
  \qquad
  \tau(s)^2=1-\frac{16}{s}+O(s^{-2}),
  \qquad
  \rho(s)^4=1-\frac{8}{s}+O(s^{-2}).
\end{align}
For large $\kappa_*$, the bound~\eqref{eq:local_rate} therefore
guarantees, per gradient evaluation, an asymptotic decrease twice as
large as the one of steepest descent in its worst case.
The factor $\tau(\kappa)^2$ is attained by one ME step on a quadratic
in $\R^3$ with eigenvalues $\mu$, $(\mu+L)/2$, $L$ and a suitable
starting point, because the Chebyshev polynomial
in~\eqref{eq:pi2_cheb} equioscillates at these three points.
\end{remark}

\begin{corollary}[Dimension two]\label{cor:dim2}
If $n=2$ and $v^k$, $w^k$ are linearly independent, then
$x^{k+1}=x^*$.
\end{corollary}

\begin{proof}
In this case $\Pi_k=\R^2$.
\end{proof}

For $n=2$ and parallel gradients the step is an exact line search and
need not reach $x^*$; Theorem~\ref{thm:global} still applies, and
Theorem~\ref{thm:local} gives superlinear convergence when $f$ is
twice differentiable at $x^*$, since a $2\times 2$ matrix has at most
two distinct eigenvalues.

\section{Numerical experiments}\label{sec:numerical}

We illustrate the results on the $\ell_2$-regularized logistic loss
\begin{align}\label{eq:logreg}
  f(x)=\frac{1}{m}\sum_{i=1}^{m}
       \log\bigl(1+e^{-b_ia_i^Tx}\bigr)
       +\frac{\mu}{2}\lVert x\rVert^2,
\end{align}
with $a_i\in\R^n$ the rows of a matrix $A\in\R^{m\times n}$ and
$b_i\in\{-1,+1\}$.
This function belongs to $\mathcal{F}_{\mu,L}(\R^n)$ for every
$L\geq\lambda_{\max}(A^TA)/(4m)+\mu$, and it has the structure of
Remark~\ref{rem:cost}.

\paragraph{Methods.}
We compare six methods.
\textbf{ME} is Algorithm~ME.
\textbf{ME-C} is the centers variant of~\cite{BCFG2026}, implemented
from equations~(6)--(11) there, with exact minimization on the
half-line of centers.
\textbf{ME-M} is the midpoint variant of~\cite{BCFG2026},
$x^{k+1}=\tfrac12(x^k+y^k)$.
\textbf{GD} is steepest descent with exact line search.
\textbf{AG} is Nesterov's method for strongly convex
functions~\cite[Section~2.2]{Nesterov2004} with step $1/L$ and
momentum $(\sqrt{\kappa}-1)/(\sqrt{\kappa}+1)$; it is the only method
that uses $L$ and $\mu$.
\textbf{CG} is the Polak--Ribi\`ere+ nonlinear conjugate gradient
method with exact line search~\cite[Section~5.2]{Nocedal2006}.
All methods store $Ax^k$ and evaluate restrictions of $f$ to lines
and planes as in Remark~\ref{rem:cost}.
The companion step is computed by Newton's method on
$s\mapsto f(x^k-sv^k)-f(x^k)$, started to the right of the root, and
line, half-line, and plane searches by Newton's method in one or two
variables, stopped when the Newton decrement reaches rounding level.
ME, ME-C, and ME-M needed between $6$ and $11$ inner iterations per
outer iteration on average (evaluations of the companion equation
plus Newton steps of the searches), each of cost $O(m+n)$, and the
line searches of GD and CG fewer than two Newton steps.

The cost is measured by the number of products with $A$ or $A^T$;
ME and ME-C use four per iteration and the other methods two.
All methods start at $x^1=0$ and stop when the gradient norm is at
most $10^{-6}$ (for AG, the gradient at the extrapolated point).
The value $f^*$ is computed by Newton's method to a gradient norm
below $10^{-14}$.
The code is written in Julia~\cite{Bezanson2017} (version 1.13.4) and
was run with a single BLAS thread on a Linux virtual machine with two
2.1\,GHz Intel Xeon cores and 6\,GB of memory; times are the best of
three runs for runs shorter than two seconds.
The code is available from the author's web page:
\begin{center}
\url{https://sites.google.com/site/joseyunierbellocruz}
\end{center}

\paragraph{Instances.}
The first family uses the sizes and the scaling of~\cite{BCFG2026}:
the pairs $(n,\kappa)$ are $(500,100)$, $(1000,500)$, and
$(2000,1000)$, the entries of
$A$ are standard Gaussian, the labels are uniformly random,
$L=\lVert A\rVert_F^2/(4m)+\mu$, and $\mu$ is chosen so that $L/\mu$
equals the prescribed $\kappa$; we take $m=n/2$.
For Gaussian matrices this $L$ exceeds the bound
$\lambda_{\max}(A^TA)/(4m)+\mu$ by two orders of magnitude, so the
prescribed $\kappa$ is far from the actual conditioning; we report
$\kappa_{\mathrm{s}}\coloneqq(\lambda_{\max}(A^TA)/(4m)+\mu)/\mu$ and
the condition number $\kappa_*$ of $\nabla^2f(x^*)$.
In the second family $m=5n$, the columns of a Gaussian matrix are
scaled by factors decreasing geometrically from $1$ to $1/30$, the
labels are $b=\operatorname{sign}(A\bar x+\xi)$ with Gaussian $\bar x$
and Gaussian noise $\xi$, $L=\lambda_{\max}(A^TA)/(4m)+\mu$, and
$\kappa=L/\mu\in\{10^2,10^3,10^4\}$.

\begin{table}[H]
\centering
{\small
\begin{tabular}{lrrrrrrrrr}
\toprule
& \multicolumn{3}{c}{$n=500$, $\kappa=100$}
& \multicolumn{3}{c}{$n=1000$, $\kappa=500$}
& \multicolumn{3}{c}{$n=2000$, $\kappa=1000$}\\
& \multicolumn{3}{c}{$\kappa_{\mathrm{s}}=2.12$, $\kappa_*=2.05$}
& \multicolumn{3}{c}{$\kappa_{\mathrm{s}}=3.86$, $\kappa_*=3.35$}
& \multicolumn{3}{c}{$\kappa_{\mathrm{s}}=3.89$, $\kappa_*=3.38$}\\
\cmidrule(lr){2-4}\cmidrule(lr){5-7}\cmidrule(lr){8-10}
Method & it. & prod. & time & it. & prod. & time & it. & prod. & time\\
\midrule
ME   & 5   & 22  & 0.0010 & 7   & 30  & 0.0062 & 7   & 30  & 0.022\\
ME-C & 5   & 22  & 0.0010 & 7   & 30  & 0.0061 & 7   & 30  & 0.022\\
ME-M & 11  & 24  & 0.0011 & 17  & 36  & 0.0079 & 17  & 36  & 0.028\\
GD   & 11  & 24  & 0.0010 & 17  & 36  & 0.0074 & 17  & 36  & 0.025\\
AG   & 128 & 258 & 0.0057 & 288 & 578 & 0.0929 & 412 & 826 & 0.481\\
CG   & 8   & 18  & 0.0007 & 11  & 24  & 0.0048 & 11  & 24  & 0.016\\
\bottomrule
\end{tabular}
}
\vspace{4pt}
\caption{First family ($m=n/2$, $L$ from the Frobenius norm).
  Iterations, products with $A$ or $A^T$, and time in seconds to
  reach $\lVert\nabla f\rVert\leq 10^{-6}$.}
\label{tab:nominal}
\end{table}

\begin{table}[H]
\centering
{\small
\begin{tabular}{lrrrrrrrrr}
\toprule
& \multicolumn{3}{c}{$n=200$, $\kappa=10^2$}
& \multicolumn{3}{c}{$n=500$, $\kappa=10^3$}
& \multicolumn{3}{c}{$n=1000$, $\kappa=10^4$}\\
& \multicolumn{3}{c}{$\kappa_*=47.8$}
& \multicolumn{3}{c}{$\kappa_*=332.1$}
& \multicolumn{3}{c}{$\kappa_*=1536.2$}\\
\cmidrule(lr){2-4}\cmidrule(lr){5-7}\cmidrule(lr){8-10}
Method & it. & prod. & time & it. & prod. & time & it. & prod. & time\\
\midrule
ME   & 53  & 214 & 0.033 & 312  & 1250 & 0.91 & 1238 & 4954 & 11.2\\
ME-C & 53  & 214 & 0.030 & 292  & 1170 & 0.84 & 1192 & 4770 & 10.2\\
ME-M & 188 & 378 & 0.061 & 920  & 1842 & 1.27 & 3898 & 7798 & 17.4\\
GD   & 195 & 392 & 0.046 & 1140 & 2282 & 1.22 & 4746 & 9494 & 17.3\\
AG   & 92  & 186 & 0.007 & 283  & 568  & 0.20 & 837  & 1676 & 2.45\\
CG   & 39  & 80  & 0.010 & 97   & 196  & 0.11 & 200  & 402  & 0.74\\
\bottomrule
\end{tabular}
}
\vspace{4pt}
\caption{Second family ($m=5n$, $L=\lambda_{\max}(A^TA)/(4m)+\mu$).
  Iterations, products with $A$ or $A^T$, and time in seconds to
  reach $\lVert\nabla f\rVert\leq 10^{-6}$.}
\label{tab:spectral}
\end{table}

\paragraph{Results.}
Tables~\ref{tab:nominal} and~\ref{tab:spectral} report iterations,
products, and times.
On the first family all instances are well conditioned
($\kappa_*\leq 3.4$), and ME and ME-C stop after $5$ to $7$
iterations.
AG is slow there because it is run with the overestimated constant
$L$.
On the second family ME needs between $3.6$ and $3.9$ times fewer
iterations than GD, hence between $1.8$ and $1.9$ times fewer
products, in line with Remark~\ref{rem:compare_sd}.
ME and ME-C are indistinguishable on these problems: near the
solution the objective is close to a quadratic, for which the two
iterates coincide.
ME-M behaves like GD.
AG, which uses $L$ and $\mu$, needs fewer products than ME on the
three instances, by factors $1.2$, $2.2$, and $3.0$, and CG is the
fastest method on every instance of the second family.
In no run were the two gradients parallel: $\sin^2\theta_k$ stayed in
$[0.08,0.36]$ on the first family and in $[0.62,0.82]$ on the second.

Table~\ref{tab:ratios} compares the largest observed one-step ratio
$(f(x^{k+1})-f^*)/(f(x^k)-f^*)$ with the bounds of
Theorems~\ref{thm:global} and~\ref{thm:local}.
For both ME and GD the largest ratios occur at the end of the run and
are close to the local values, $\tau(\kappa_*)^2$ for ME and
$\rho(\kappa_*)^2$ for GD, within $2.1\%$ and $0.6\%$ respectively;
the global bound $\rho(\kappa)^2$ is not approached.
Figure~\ref{fig:convergence} shows the optimality gap against the
number of products for the second instance.

\begin{table}[H]
\centering
{\small
\begin{tabular}{rrccccc}
\toprule
& & \multicolumn{3}{c}{ME} & \multicolumn{2}{c}{GD}\\
\cmidrule(lr){3-5}\cmidrule(lr){6-7}
$\kappa$ & $\kappa_*$ & observed & $\tau(\kappa_*)^2$
  & $\rho(\kappa)^2$ & observed & $\rho(\kappa_*)^2$\\
\midrule
$10^2$ & $47.8$   & $0.7100$ & $0.7247$ & $0.9608$ & $0.9147$ & $0.9197$\\
$10^3$ & $332.1$  & $0.9499$ & $0.9532$ & $0.9960$ & $0.9867$ & $0.9880$\\
$10^4$ & $1536.2$ & $0.9884$ & $0.9897$ & $0.9996$ & $0.9971$ & $0.9974$\\
\bottomrule
\end{tabular}
}
\vspace{4pt}
\caption{Second family: largest observed ratio
  $(f(x^{k+1})-f^*)/(f(x^k)-f^*)$ over the run, compared with the
  local bound of Theorem~\ref{thm:local}, the global bound of
  Theorem~\ref{thm:global}, and the factor $\rho(\kappa_*)^2$ of
  steepest descent.}
\label{tab:ratios}
\end{table}

\begin{figure}[H]
\centering
\begin{tikzpicture}
\begin{semilogyaxis}[
  width=0.82\columnwidth, height=0.5\columnwidth,
  xlabel={products with $A$ or $A^T$},
  ylabel={$f(x^k)-f^*$},
  xmin=0, xmax=2400, ymin=1e-10, ymax=1,
  legend pos=north east, legend cell align=left,
  legend style={font=\small}, grid=major, grid style={gray!25}]
\addplot[blue, thick] table[x=mv, y=gap] {fig_gap_ME.dat};
\addlegendentry{ME}
\addplot[blue, thick, dashed] table[x=mv, y=gap] {fig_gap_MEC.dat};
\addlegendentry{ME-C}
\addplot[red, thick, dashdotted] table[x=mv, y=gap] {fig_gap_GD.dat};
\addlegendentry{GD}
\addplot[black!50!green, thick, dotted] table[x=mv, y=gap] {fig_gap_AG.dat};
\addlegendentry{AG}
\addplot[black, thick, densely dashed] table[x=mv, y=gap] {fig_gap_CG.dat};
\addlegendentry{CG}
\end{semilogyaxis}
\end{tikzpicture}
\caption{Optimality gap against the number of products with $A$ or
  $A^T$ on the instance $n=500$, $m=2500$, $\kappa=10^3$ of the
  second family.}
\label{fig:convergence}
\end{figure}

\section{Concluding remarks}\label{sec:conclusion}

For $L$-smooth, $\mu$-strongly convex functions, the plane-search form
of the Method of Ellipcenters has the one-step factor
$\rho^2-\sin^2\theta_k/(4\kappa^2)$ at every iterate, and the two-step
Chebyshev factor of the Hessian at the solution in the limit.
The first statement places the method at the level of steepest
descent with exact line search in the worst case, at twice the cost
per iteration; the second places it at the level of two conjugate
gradient steps near the solution.
The experiments agree with this description: the method is about twice
as fast as steepest descent in products with the data matrix, and it
is not competitive with a conjugate gradient method on
ill-conditioned problems.

Several questions remain open.
The computation of Section~\ref{sec:pep} indicates that the worst-case
factor $B(\kappa)$ of one step satisfies
$\rho^4<B(\kappa)<\rho^2$; a closed form and a proof are missing, and
so is an explanation of the change of regime observed between
$\kappa=10$ and $\kappa=20$.
The analysis assumes exact plane searches; a version of
Theorem~\ref{thm:global} for inexact searches, in the spirit of the
noisy line search of~\cite{deKlerk2017}, would cover objectives
without the structure of Remark~\ref{rem:cost}.
Finally, for quadratics the addition of a momentum direction to the
plane leads to the method of~\cite{BGO2026}; the corresponding
three-dimensional search for general strongly convex functions is a
natural next step.

\paragraph{Acknowledgements.}
The author thanks Roger Behling and Luiz-Rafael Santos
for ongoing discussions on the Method of Ellipcenters for quadratic functions,
which began approximately eight years before the
appearance of~\cite{ME2025}.

\paragraph{Disclosure statement.}
No potential conflict of interest was reported by the author.

\end{document}